\documentclass[a4paper,12pt]{amsart}

\usepackage{amsmath}
\usepackage{amsfonts}
\usepackage{amssymb}
\usepackage{graphicx}
\usepackage{hyperref}
\usepackage{color}
\usepackage{soul,xcolor} 
\usepackage{cite}

\usepackage{amsthm}
\usepackage[all,cmtip]{xy}

\newcommand{\mbQ}{\mathbb{Q}}

\newcommand{\mbZ}{\mathbb{Z}}
\newcommand{\mbP}{\mathbb{P}}

\DeclareMathOperator{\N}{N}

\DeclareMathOperator{\NE}{\overline{NE}}
\DeclareMathOperator{\Exc}{Exc}

\newcommand*{\coloneq}{\mathrel{\mathop:}=}

\theoremstyle{plain}
\newtheorem{theorem}{Theorem}[section]
\newtheorem{proposition}[theorem]{Proposition}
\newtheorem{lemma}[theorem]{Lemma}

\newtheorem{claim}[theorem]{Claim}

\theoremstyle{definition}
\newtheorem{definition}[theorem]{Definition}

\theoremstyle{remark}
\newtheorem{remark}[theorem]{Remark}

\title[Effective bounds for the number of MMP-series]
{Effective bounds for the number of MMP-series of a smooth threefold}

\author{Diletta Martinelli}
\address{
	Korteweg-des Vries Institute for Mathematics, Universiteit van Amsterdam, P.O.Box 94248, 1090 GE, Amsterdam}
\email{d.martinelli@uva.nl}

\begin{document}
	\bibliographystyle{alpha}

\begin{abstract}
	We prove that the number of MMP-series of a smooth projective threefold of positive Kodaira dimension and of Picard number equal to three is at most two.  
\end{abstract}

\subjclass[2010]{Primary 14E30; Secondary 14J30}
\keywords{minimal model program, flip, Picard number}

\maketitle


\section{Introduction}

Establishing the existence of minimal models is one of the first steps towards the birational classification of smooth projective varieties. Moreover, starting from dimension three, minimal models are known to be non-unique, leading to some natural questions such as: does a variety admit a finite number of minimal models? And if yes, can we fix some parameters to bound this number? 

We have to be specific in defining what we mean with minimal models and how we count their number. A \emph{marked} minimal model of a variety $X$ is a pair $(Y,\phi)$, where $\phi \colon X\dashrightarrow Y$ is a birational map and $Y$ is a minimal model. The number of marked minimal models, up to isomorphism, for a variety of general type is finite \cite{BCHM10, KM87}. When $X$ is not of general type, this is no longer true, \cite[Example 6.8]{Rei83}.  It is, however, conjectured that the number of minimal models up to isomorphism is always finite and the conjecture is known in the case of threefolds of positive Kodaira dimension \cite{Kaw97}.

In \cite{MST16} it is proved that it is possible to bound the number of marked minimal models of a smooth variety of general type in terms of the canonical volume. Moreover, in dimension three Cascini and Tasin \cite{CT18} bounded the volume using the Betti numbers. This result can be used to show that the number of marked minimal models of a threefold of general type can be bounded using topological data \cite[Corollary 3]{MST16}, solving a conjecture of Cascini and Lazi{\'c} \cite{CL14}.

In this note we address a closely related, although different, question. In the framework of the Minimal Model Program (MMP for short), starting from a smooth projective threefold $X$, in order to find a minimal model of $X$, we need to perform contractions of $K_X$-negative extremal rays. There are two types of birational contractions, either the exceptional locus is a divisor, then we call the contraction \emph{divisorial}; or the exceptional locus has codimension greater than or equal to two, in this case the target of the contraction is no longer $\mathbb{Q}$-factorial and in order to proceed in the program we need to perform a birational surgery called \emph{flip}, see Definition \ref{def:flip}. We call a series of divisorial contractions and flips an \emph{MMP-series}.  In the case of a smooth projective threefold $X$ of positive Kodaira dimension and Picard number equal to three
we bound in an effective way the possible MMP-series for $X$. This is the first non-trivial case, see Lemma \ref{lem:pic2}. 
\begin{theorem} \label{cor:dile1}
	Let $X$ be a smooth projective threefold of positive Kodaira dimension and of Picard number $\rho(X) = 3$, then the number of MMP-series for $X$ is at most two.
\end{theorem}
Explicit examples of MMP-series for a threefold as in Theorem \ref{cor:dile1} can be found in Mori's classification of divisorial contractions for a smooth projective threefold, see \cite[Theorem 3.3, Corollary 3.4, Theorem 3.5]{Mor82}. The target of the first divisorial contraction is Gorenstein, in all but one case where the contraction is the blow up over a terminal quotient singularity of type $\frac{1}{2}(1,1,1)$, see \cite{Kaw96}. In this case it is possible, after a flip, to reach a minimal model with Picard number equal to two and difficulty equal to zero (see Definition \ref{def:difficulty} for the definition of difficulty).  
\\

Theorem \ref{cor:dile1} reduces quickly to finding a bound for the number of possible flipping curves passing through a terminal singularity, see \cite[Definition 2.34]{KM98}. Note that in the case of a smooth surface $S$ of positive Kodaira dimension, the minimal model is unique \cite[Definition 1.30]{KM98} and, therefore, we cannot have two (--1)-curves $E_1$ and $E_2$ passing through the same point. Indeed, they both should be contracted before reaching the minimal model of $S$, but the contraction of $C_1$ transforms $C_2$ into a curve with non-negative selfintersection, and so impossible to contract.

The difficulty in generalizing this result to higher Picard rank comes from the fact that we cannot control the number of flipping curves contained in a divisor that is later contracted in the MMP-series. In the non-general type case there exists an example of a threefold $X$ where this number of curves is infinite, hence producing a new example of a variety with non-negative Kodaira dimension and an infinite number of $K_X$-negative extremal rays, see \cite[Theorem 7.1]{Les15}. In the general type case this cannot happen because of the finiteness statement proved in \cite{BCHM10}, and there is still hope to bound the total number of MMP-series using the topology of the variety. Also, with the current techniques it does not seem possible to relax the smoothness hypothesis in Theorem \ref{cor:dile1}, see Remark \ref{rmk:singul3}. In the case of Picard number equal two, instead, we can give an easy bound also in the case of terminal threefolds, see Remark \ref{rmk:pic2terminal}.


The paper is organized as follows: in Section \ref{sec:prelim} we collect some preliminary notions, mainly about the MMP in dimension three. The reader in need of more details should refer to \cite{KM98}. In Section \ref{sec:proof}
we prove Theorem \ref{cor:dile1}.

\section*{Acknowledgements}
This work is part of my Ph.D. thesis. I would like to thank my supervisor Paolo Cascini for proposing me to work on this problem and for constantly supporting me with his invaluable help and comments. I also would like to thank Ivan Cheltsov, Alessio Corti, Claudio Fontanari, Enrica Floris, Stefan Schreieder and Luca Tasin for the many fruitful conversations we had about this subject.

The results of this paper were conceived when I was supported by a Roth Scholarship.

\section{Preliminary Results} \label{sec:prelim}
%
All the varieties that we are going to consider are assumed to be projective and defined over the field of the complex numbers. 
\\

We will always refer to a MMP-series as a series of divisorial contractions and flips performed during a $K_X$-negative MMP. In this context, a minimal model for a variety $X$ is just an output that admits no further MMP-series.
\\

The varieties appearing in the steps of a MMP-series for a smooth threefold are projective $\mbQ$-factorial varieties characterized by a type of mild singularities called terminal, see \cite[Definition 2.34]{KM98}.

\subsection{The Picard number}
Let $X$ be a normal variety. Two Cartier divisors $D_1$ and $D_2$ on $X$ are numerically equivalent, $D_1 \equiv D_2$, if they have the same degree on every curve on $X$, i.e. if $D_1 \cdot C = D_2 \cdot C$ for each curve $C$ in $X$. The quotient of the group of Cartier divisors modulo this equivalence relation is denoted by $\N^1(X)$.

We can also define $\N^1(X)$ as the subspace of cohomology $H^2(X, \mbZ)$ spanned by divisors. 
%
We write $\N^1(X)_{\mbQ} \coloneq \N^1(X) \otimes_{\mathbb{Z}} \mathbb{Q}$.

\begin{definition} \label{def:picard}
	We define $\rho(X) \coloneq \dim_{\mbQ} \N^1(X)_{\mbQ}$ and we call it the Picard number of $X$.
\end{definition}
We remark that $\rho(X) \leq b_2$, the second Betti number of $X$, that depends only on topological information of $X$.

Similarly, two $1$-cycles $C_1$ and $C_2$ are numerically equivalent if they have the same intersection number with any Cartier divisor. We call $\N_1(X)$ the quotient group and we write $\N_1(X)_{\mbQ} \coloneq \N_1(X) \otimes_{\mathbb{Z}} \mathbb{Q}$. We can also see $\N_1(X)$ as the subspace of homology $H_2(X, \mbZ)$ spanned by algebraic curves. See for details \cite[Section 1.4]{Deb13}. We denote by $\NE(X)$ the cone of $\N_1(X)_{\mbQ}$ generated by effective 1-cycles and we call it the cone of curves on $X$, see for instance \cite[Section 1.3]{KM98}. 
\\

We will also use the following defintion.
\begin{definition} \label{def:rhoRestr}
	Let $E$ be a subscheme contained in a projective variety $X$. We consider the following natural map
	\begin{equation*}
	\psi: \NE(E)_{\mbQ} \longrightarrow \NE(X)_{\mbQ}
	\end{equation*}
	and we denote  
	$\NE(E|X)_{\mbQ} \coloneq \psi(\NE(E)_{\mbQ}) \subseteq \NE(X)_{\mbQ}$,
	and $\rho(E|X) \coloneq \dim \NE(E|X)_{\mbQ}$.
	Note that $\ker{\psi}$ might be non-trivial.
	
\end{definition}

\subsection{The difficulty}

In dimension three, the existence and termination of flips was proved by Mori and Shokurov. A key ingredient for the proof of termination is the so called \emph{difficulty} of $X$, introduced by Shokurov.

\begin{definition} \cite[Definition 2.15]{Sho86}, \cite[Definition 6.20]{KM98} \label{def:difficulty}
	Let $X$ be a projective normal $\mathbb{Q}$-Gorenstein threefold. The \emph{difficulty} of $X$ is defined as follows	\begin{equation*}
	d(X) \coloneq \#\{ E  \text{ prime divisor } | \\\ a(E, X) < 1, E \text{  is exceptional over } X\}, 
	\end{equation*}
	where $a(E, X)$ is the discrepancy of $E$ with respect to $X$. 
\end{definition}

\begin{remark}\label{rmk:difficulty}
	Note that the difficulty always goes down under a flip, and if $X$ is smooth, then $d(X) = 0$ and we cannot have any flips. See \cite[Lemma 3.38]{KM98}.
\end{remark}


We also recall that minimal models are connected by a sequence of flops, \cite[Theorem 4.9]{Kol89}.
\\

We recall here the definition of a flip, since it is central in the proof of Theorem \ref{cor:dile1}.

\begin{definition} \cite[Definition 3.33]{KM98} \label{def:flip}
Let $X$ be a normal scheme. A flipping contraction is a proper birational morphism $f \colon X \to Y$ to a normal scheme $Y$ such that the exceptional locus $\Exc(f)$ has codimension at least two in $X$ and $-K_X$ is $f-$ample. A normal scheme $X^+$ together with a proper birational morphism $f^+ \colon X^+ \to Y$ is called a \emph{flip} of $f$ if
\begin{itemize}
\item $K_{X^+}$ is $\mathbb Q$-Cartier.
\item $K_{X^+}$ is $f^+$-ample.
\item $\Exc(f^+)$ has codimension at least two in $X^+$.
\end{itemize} 

By slight abuse of terminology, we also call the induced rational map $\phi \colon X \dashrightarrow X^+$ a flip. A flip gives rise to the following commutative diagram.

\[\xymatrixcolsep{3pc}\xymatrix{
		X \ar[dr]_{f}  \ar@{-->}[rr]^{\phi} & &X' \ar[dl]^{f^+}  \\
		& Y& \\
	}\]
A curve $C \in \Exc(f)$ is called a \emph{flipping} curve and $C^+ \in \Exc(f^+)$ is called a \emph{flipped} curve.	
	
\end{definition}

For the definitions of divisorial contraction and flop we refer to \cite[Proposition 2.5, Definition 6.10]{KM98}. We just recall that under a divisorial contraction $f \colon X \dashrightarrow X'$ the Picard number drops by one, i.e. $\rho(X') = \rho(X) - 1$. If $f$ is a flip instead, the Picard number does not change, i.e $\rho(X') = \rho(X)$.

\section{Proof of Theorem \ref{cor:dile1}} \label{sec:proof}

In this section, we prove Theorem \ref{cor:dile1}.
The strategy of the proof is to first bound the number of steps of the MMP-series and then count how many are the possible divisorial contractions to a point or to a curve and how many flips con appear in a MMP-series.
\\

	Our starting point is a smooth projective threefold $X$ of positive Kodaira dimension. As we recalled in Remark \ref{rmk:difficulty}, this means that the difficulty of $X$ has to be zero and no flips are possible. Then, the first operation of the MMP-series for $X$ has to be a divisorial contraction.
	If $\rho(X) = 1$, no contractions are possible. 
	
	\begin{lemma} \label{lem:pic2}
		Let $X$ be a smooth projective threefold of non-negative Kodaira dimension such that $\rho(X) = 2$. Then there is only a unique MMP-series for $X$.
	\end{lemma}
	
	\begin{proof}
	We remark that, $X$ being smooth, the first operation in the MMP-series needs to be a divisorial contraction $\phi' \colon X \to X'$. Since the Picard number has dropped by one, $\rho(X') = 1$, and so $X'$ admits neither flipping nor flopping contractions, and no flips nor flops are possible. In particular, $X'$ is minimal. If there were another divisorial contraction $\phi'' \colon X \to X''$, then two minimal models $X'$ and $X''$ would be connected by a sequence of flops \cite[Theorem 4.9]{Kol89}, a contradiction. We conclude that $\phi'$ is the only MMP-series for $X$.
%
	\end{proof}
%
	\begin{remark} \label{rmk:endPoint}
	Let $X'$ be a minimal model for $X$, then either $X'$ is unique or we can assume $\rho(X') \geq 2$. Indeed, if it were $\rho(X') = 1$ then $X'$ would not admit any flopping contractions and so no flops are possible. We then conclude that $X'$ is the unique minimal model of $X$ since different minimal models are related via flops.
\end{remark}

\begin{remark} \label{rmk:pic2terminal}
If we relax the smoothness hypothesis and we assume that $X$ has terminal singularities and $\rho(X) = 2$, the difficulty of $X$ might be bigger that zero. In particular, we can assume that $d(X) = k$, where $k$ is a natural number. If we consider a MMP-series for $X$, at any step there can be either a divisorial contraction or a flip. If there is a divisorial contraction, the target of the contraction will be a variety $Y$ with $\rho(Y) = 1$ and so we reach an end point, see Remark \ref{rmk:endPoint}. Hence, we are left to count how many flips we can have at any step: since the Picard number is two, at any step we have two extremal rays and so at most two possible flips. We conclude that we have at most $2^k$ possible MMP-series for $X$.
\end{remark}

Let us proceed now with the bound for the number of steps. It is a calculation that follows from the termination of flips in dimension three, see \cite[Lemma 3.1]{CZ14}.

\begin{lemma} \label{lem:steps} 
	Let $X$ be a projective threefold with terminal singularities such that $\rho(X) \geq 3$ and let $X'$ be the outcome of a MMP-series for $X$. We suppose in addition that $\rho(X') \geq 2$. Let $l$ be the length of a MMP-series of $X$. Then $l$ is at most $2(\rho(X) - 2)$.
\end{lemma}

\begin{proof}
	We denote by $l_D$ the total number of divisorial contractions and by $l_F$ the total number of flips. 
	Clearly, $l = l_D + l_F$.

	Under a divisorial contraction the Picard number drops by one. Hence, $l_D = \rho(X) - \rho(X') \leq \rho(X) - 2$. To conclude the proof, we claim that $l_F \leq l_D$. Under a flip, the Picard number is stable and we need to consider the difficulty $d(X)$, see Definition \ref{def:difficulty}. 
	If $X$ is smooth,  $d(X) = 0$ and no flips are possible, see Remark \ref{rmk:difficulty}. Moreover, if $X_{i-1} \to X_i$ is a divisorial contraction, then
	\begin{equation*}
	d(X_i) \leq d(X_{i - 1}) + 1,
	\end{equation*}
	see for instance \cite[Lemma 3.1]{CZ14}.
	Otherwise, if $X_{i-1} \dashrightarrow X_i$ is a flip, then
	\begin{equation*}
	d(X_i) \leq d(X_{i - 1}) - 1,
	\end{equation*}
	because flips strictly improve the singularities (see \cite[Lemma 3.38, Definition 6.20, Lemma 6.21]{KM98}). We conclude that $d(X') \leq d(X) + l_D - l_F$. Thus, $ l_F \leq l_D$ and $l \leq 2(\rho(X) - 2)$.
\end{proof}

We now proceed by counting how many choices we have for the series of divisorial contractions.

\begin{lemma} \label{lem:divContr1}
	Let $X$ be a smooth projective threefold of non-negative Kodaira dimension and let $\phi' \colon X \dashrightarrow X'$ be a MMP-series for $X$. Then there are at most $2^{\rho(X) - 2}(\rho(X) - 2)!$ choices for the series of divisorial contractions from $X$ to $X'$.
\end{lemma}

\begin{proof}
	We recall that we denote by $l_D$ the total number of divisorial contractions in a MMP-series for $X$. Let $\{E_i'\}_{1 \leq i \leq l_D}$  be the set of prime divisors on $X$ which are exceptional over $X'$. Let $\phi'' \colon X \dashrightarrow X''$ be a MMP-series for $X$, different from $\phi'$.
	\begin{claim}
		The set $\{E_i'\}_{1 \leq i \leq l_D}$ is independent of the choice of $\phi'$. 
	\end{claim}
\begin{proof}[Proof of the Claim]
	 Let $E''$ be a prime divisor on $X$ which is exceptional over $X''$. Assume that $E'' \neq E_i$ for every $1 \leq i \leq l_D$. Then $E''$ is not exceptional over $X'$, so in particular it is not contracted by $\phi'$ and $\phi'(E'')$ is a divisor in $X'$. Since, $X'$ and $X''$ are minimal models, we know that they are connected via a sequence of flops $\eta$, 
	see \cite[Theorem 4.9]{Kol89}. We have, therefore, the following diagram.
	
	\[\xymatrixcolsep{5pc}\xymatrix{
		X \ar[dr]_{\phi''}  \ar[r]^{\phi'} & X' \ar@{-->}[d]^{\eta}  \\
		& X'' \\
	}\]
	Since $E''$ is exceptional over $X''$, we obtained that $\phi'(E'')$ needs to be contracted by $\eta$, but this is a contradiction because flops are isomorphisms in codimension one. 
\end{proof}
Then the series of divisorial contractions from $X$ to $X'$ just contracts $\{E_i'\}_{1 \leq i \leq l_D}$ one by one at each time. Respecting the order of contractions of $\{E_i'\}_{1 \leq 1 \leq l_D}$, there are at most $l_D! \leq (\rho(X) - 2)!$ choices.
\\

Let $E$ be a prime divisor contracted by an extremal divisorial contraction in the MMP-series $\phi' \colon X \dashrightarrow X'$. We can conclude the proof thanks to Lemma \ref{lem:waysDiv} that shows that it is possible to contract $E$ in at most two ways.
\end{proof}

\begin{lemma} \label{lem:waysDiv}
	Let $Y$ be a projective threefold with terminal singularities. Let $E \subseteq Y$ be a prime divisor that can be contracted by a $K_Y$-negative extremal contraction. Then there are at most two ways to contract $E$.
\end{lemma}

\begin{proof}
	Let $R$ be a $K_Y$-negative extremal ray in the cone of curves $\NE(Y)$ and let $g_R \colon Y \to Y'$ be the associated contraction. Let us assume that $E$ is contracted by $g_R$. If $g_R$ contracts $E$ to a point, then it means that all the equivalence classes of curves on $E$ are numerically proportional to each other in $Y$, i.e. $\rho(E|Y) = 1$, see Definition \ref{def:rhoRestr}. If $g_R$ contracts $E$ to a curve $C$ instead, we have the following diagram

	\[\xymatrixcolsep{5pc}\xymatrix{
     \NE(E|Y)_{\mbQ} \ar@{^{(}->}[d] \ar[r]^{\beta} & \NE(C|Y')\cong\NE(C)\cong \mathbb Z \ar@{^{(}->}[d]\\
	 \NE(Y) \ar[r]^{\alpha} & \NE(Y') \\
	}\]
where $\alpha$ has kernel of dimension one because it is an extremal contraction, therefore the same holds for $\beta$ and so $\rho(E|Y) = 2$ and there can be at most two extremal rays that generate the contraction.  	
%
\end{proof}

\begin{remark}
	The simplest example of a prime divisor $E$ contracted to a curve in two different ways is the case of Atiyah's flop, see for instance \cite[Example 1.16]{HM10}, where $E \cong \mathbb{P}^1 \times \mbP^1$. See \cite{MP20} and \cite{Tzi05} and references therein for more examples and a partial classification of contractions of a divisor to a curve on a terminal threefold.
	
	
\end{remark}
\subsection{Bounding the last flip}
We now prove that if the last operation of the MMP is a flip it can then be obtained in a unique way. We refer to Definition \ref{def:flip} for the definition and notation of a flip.

\begin{proposition} \label{prop:flips}
	Let $X$ be a projective threefold of positive Kodaira dimension with terminal singularities. Let $\psi \colon X \dashrightarrow X_{\text{min}}$ be a flip such that $X_{\text{min}}$ is a minimal model for $X$. Then $\psi$ is the only possible flip.
\end{proposition}
\begin{proof}
Assume by contradiction that there exists another flip $\psi'$ different from $\psi$. Hence $X$ fits into the following commutative diagram.

		\[\xymatrixcolsep{3pc}\xymatrix{
		X' \ar[dr]_{f'^+}& & X \ar[dr]_{f} \ar[dl]^{f'}  \ar@{-->}[rr]^{\psi}  \ar@{-->}[ll]_{\psi'} & & X_{\text{min}} \ar[dl]^{f^+}  \\
		& Y'& & Y& \\
	}\]
Since $\psi$ and $\psi'$ are different flips, there exists at least a curve $\xi'$ that is a flipping curve for $\psi'$ but not for $\psi$: i.e. $\xi' \subseteq  \Exc(f')$ but $\xi' \not\subseteq  \Exc(f)$. In particular, $\xi'$ belongs to the locus where $\psi$ is an isomorphism and $\psi(\xi') \subseteq  X_{\text{min}}$.

	Thanks to Abundance Theorem \cite{Kol92}, there exists an integer $m$ such that $|mK_{X_{\text{min}}}|$ is base point free. Therefore, we can choose a general surface
	\begin{equation*} \label{eq:Schoice}
	S' \in |mK_{X_{\text{min}}}|
	\end{equation*}
	in such a way that it does not contain $\psi(\xi') \subseteq  X_{\text{min}}$. Notice that the hypothesis that $X$ is of positive Kodaira dimension is used here. Then let $S \coloneq (\psi'^{-1})_*(S')$ be the strict transform of $S'$. Since flips are isomorphisms in codimension one, $S \in |mK_{X}|$. Since $\xi' \subseteq \Exc(f')$ and $-K_X$ is $f'$-ample, we have that $K_X \cdot \xi' < 0$, so in particular $\xi' \subseteq  S$. Since $\psi(\xi')$ belongs to $X_{\text{min}}$ this implies that $\psi_{|S}(\xi')$ belongs to $S'$. By the definition of $S'$, we reach a contradiction.
	
%
%
%
%
%
\end{proof}

\begin{remark} \label{rmk:2flips}
We cannot hope to iterate the proof of Proposition \ref{prop:flips} to bound the number of MMP-series containing more that one flip. In order to do that we need to bound the number of extremal rays in the last but one flip, and the information that we have about the last flip coming from Proposition \ref{prop:flips} do not give insights on the problem. For instance, assume to be in the following situation 
\[\xymatrixcolsep{3pc}\xymatrix{
		X \ar[dr]_{f_1} \ar@{-->}[rr]^{\psi_1} & & X_1 \ar[dr]_{f_2} \ar[dl]^{f_1^+}  \ar@{-->}[rr]^{\psi_2}  & & X_{\text{min}} \ar[dl]^{{f_2}^+}  \\
		& Y'& & Y& \\
	}\]
and assume that there are $k$ distinct $K_X$-negative extremal rays $C_1, \dots, C_k$ on $X$. Let $\psi_1$ be the flip of $C_1$, i.e. $C_1 \subseteq \Exc(f_1)$ and $C_2, \dots, C_k \not\subseteq \Exc(f_1)$. We know, thanks to Proposition \ref{prop:flips}, that $\psi_1(C_2), \dots, \psi_1(C_k)$ are all contained in the $\Exc(f_2)$ and so they are, in particular, all numerically equivalent, but we can not use this fact to impose a bound on the number $k$.
\end{remark}

%
%
\subsection{Proof of Theorem \ref{cor:dile1}}

Now we have all the ingredients to prove Theorem \ref{cor:dile1}.
%
%

\begin{proof}[Proof of Theorem \ref{cor:dile1}]
	Let $X$ be a smooth projective threefold of positive Kodaira dimension, such that $\rho(X) = 3$. Let $X_{\text{min}}$ be a minimal model for $X$. We can assume that $\rho(X_{\text{min}}) \geq 2$, because otherwise $X_{\text{min}}$ is unique, see Remark \ref{rmk:endPoint}. In this case the graph of the possible MMP-series for $X$ is extremely simple: using Lemma \ref{lem:steps} we have at most two steps in the MMP-series, the first operation is a divisorial contraction followed by a flip. Applying Lemma \ref{lem:divContr1} we have that the first divisorial contraction can be obtained in at most two ways and then we conclude that the last flip can be realized in a unique way thank to Proposition \ref{prop:flips}. This concludes the proof.
%
\end{proof}

\begin{remark} \label{rmk:singul3}
It is important to notice that, even though several of the preliminary results also hold in the case of a terminal threefold, the proof of Theorem \ref{cor:dile1} holds only for a smooth threefold. Indeed, even if we consider the simplest case possible of a terminal threefold $Y$ such that $\rho(Y) = 3$ and $d(Y) = 1$, we could have an MMP-series composed by a flip, a divisorial contraction and a final flip and with the techniques contained in this paper it is not possible to bound MMP-series containing more that one flip, see Remark \ref{rmk:2flips}.
\end{remark}

\end{document}